\documentclass[11pt,reqno]{amsart}

\usepackage{amssymb}
\usepackage{amsmath}
\usepackage{amsthm}
\usepackage{latexsym}
\usepackage{amsfonts}
\usepackage{mathrsfs}
\usepackage{bbm}
\usepackage{titlesec}
\usepackage{color}
\usepackage{hyperref}

\newtheoremstyle{neu_thm}
{13pt}       
{8pt}      
{\itshape}  
{}          
{\bfseries} 
{.}         
{.5em}      
{}          

\newtheoremstyle{neu_defn}
{13pt}       
{8pt}      
{}  
{}          
{\bfseries} 
{.}         
{.5em}      
{}          

\theoremstyle{neu_thm}
\newtheorem{thm}{Theorem}[section]
\newtheorem{cor}[thm]{Corollary}
\newtheorem{lem}[thm]{Lemma}
\newtheorem{prop}[thm]{Proposition}

\theoremstyle{neu_defn}
\newtheorem{defn}[thm]{Definition}
\newtheorem{rem}[thm]{Remark}
\newtheorem{ex}[thm]{Example}

\titleformat{\section}{\normalfont\bfseries\centering}{\thesection.}{.25em}{}
\titleformat{\subsection}{\normalfont\bfseries}{\thesubsection.}{.25em}{}
\titlespacing{\section}{0pt}{*3}{*1.5}
\titlespacing{\subsection}{0pt}{*4}{*0.5}
\numberwithin{equation}{section}
\allowdisplaybreaks

\newcommand{\R}{\ensuremath{\mathbb R}}    
\newcommand{\C}{\ensuremath{\mathbb C}}    
\newcommand{\N}{\ensuremath{\mathbb N}}    
\newcommand{\Z}{\ensuremath{\mathbb Z}}    
\newcommand{\D}{\ensuremath{\mathbb D}}    
\newcommand{\T}{\ensuremath{\mathbb T}}    

\newcommand{\<}{\langle}
\renewcommand{\>}{\rangle}

\newcommand{\aproduct}{\langle\cdot\,,\cdot\rangle}

         \newcommand{\frakB}{\mathfrak B}
         
\newcommand{\calD}{\mathcal D}         
         
\newcommand{\calF}{\mathcal F}         
         
\newcommand{\calH}{\mathcal H}

\newcommand{\calP}{\mathcal P}

\newcommand{\calS}{\mathcal S}

\newcommand{\scrL}{\mathscr L}

\newcommand{\la}{\lambda}
\newcommand{\veps}{\varepsilon}
\newcommand{\vphi}{\varphi}

\renewcommand{\Re}{\operatorname{Re}}
\newcommand{\linspan}{\operatorname{span}}

\newcommand{\dom}{\operatorname{dom}}

\newcommand{\Sra}{\Rightarrow}

\newcommand{\Slra}{\Leftrightarrow}

\newcommand{\upto}{\uparrow}
\newcommand{\downto}{\downarrow}

\newcommand{\ol}{\overline}

\newcommand{\wh}{\widehat}

\definecolor{darkgreen}{rgb}{0,0.6,0.1}

\newcommand{\supp}{\operatorname{supp}}

\newcommand{\sinc}{\operatorname{sinc}}
\newcommand{\frake}{\mathfrak e}
\newcommand{\braces}[1]{{\rm (}#1{\rm )}}
\newcommand{\arc}{\operatorname{arc}}
\newcommand{\Leb}{\operatorname{Leb}}
\newcommand{\diag}{\operatorname{diag}}
\newcommand{\one}{\mathbbm{1}}
\DeclareMathOperator*{\liim}{l.i.m.}
\newcommand{\BL}{\operatorname{BL}}
\newcommand{\essinf}{\operatorname{essinf}}
\newcommand{\esssup}{\operatorname{esssup}}

\begin{document}
\title[]{Bessel Orbits of Normal Operators}

\author{Friedrich Philipp}
\address{Departamento de Matem\'atica, Facultad de Ciencias Exactas y Naturales, Universidad de Buenos Aires, Ciudad Universitaria, Pabell\'on I, 1428 Buenos Aires, Argentina}
\email{fmphilipp@dm.uba.ar}
\urladdr{http://cms.dm.uba.ar/members/fmphilipp/}

\begin{abstract}
Given a bounded normal operator $A$ in a Hilbert space and a fixed vector $x$, we elaborate on the problem of finding necessary and sufficient conditions under which $(A^kx)_{k\in\N}$ constitutes a Bessel sequence. We provide a characterization in terms of the measure $\|E(\cdot)x\|^2$, where $E$ is the spectral measure of the operator $A$. In the separately treated special cases where $A$ is unitary or selfadjoint we obtain more explicit characterizations. Finally, we apply our results to a sequence $(A^kx)_{k\in\N}$, where $A$ arises from the heat equation. The problem is motivated by and related to the new field of Dynamical Sampling which was recently initiated by Aldroubi et al. in \cite{acmt}.
\end{abstract}

\subjclass[2010]{94A20, 47B15, 30H10, 47B35}
\keywords{Bessel sequence, Dynamical Sampling, Hardy space, Carleson embedding theorem, Hankel matrix, Toeplitz matrix}

\maketitle
\thispagestyle{empty}


\section{Introduction}
Many signals in nature obey a differential equation of the type $\frac{\partial f}{\partial t} = Tf$ with a linear operator $T$. Dynamical Sampling incorporates this knowledge aiming to recover $f$ from spatial subsamples at several times. A prominent example is that of sensors in a forest measuring the temperature in order to prevent or detect forest fire. Usually, many sensors are needed to retrieve an accurate temperature distribution at a specific point of time. Making use of the heat equation as the time evolution law in the background, the hope is to install less sensors but to measure sequentially at different times, which is -- of course -- economically more efficient.

Let us briefly motivate the general Dynamical Sampling setting. Assume that $T$ is a bounded operator in a Hilbert space $\calH$ of functions. If we put $u(t) := f(t,\cdot)$, the differential equation reduces to $\dot u(t) = Tu(t)$ and has the solutions $u(t) = e^{tT}u_0$. Hence, sampling the functions $f(t,\cdot)$ at times $t = 0,1,2,\ldots$ is the same as sampling $u_0$, $Bu_0$, $B^2u_0$, etc., where $B = e^T$. Here, we assume that ``sampling'' $g\in\calH$ means to take scalar products between $g$ and functions $x_i$, $i\in I$, from a fixed system\footnote{If $\calH$ is a reproducing kernel Hilbert space with reproducing kernel $x_t$, then this type of sampling obviously coincides with function evaluation.}. Hence, the measurements have the form $\<B^ku_0,x_i\> = \<u_0,A^kx_i\>$, where $A = B^*$. Since we are aiming at recovering $u_0$ stably from these measurements, we require $(A^kx_i)_{i\in I,\,k=0,\ldots,K_i}$ to be a frame for some numbers $K_i\in\N\cup\{\infty\}$, $i\in I$. Dynamical Sampling is about finding necessary and sufficient conditions on the operator $A$, the vectors $x_i$, and the numbers $K_i$ under which $(A^kx_i)_{i\in I,\,k=0,\ldots,K_i}$ is a frame for $\calH$.

Motivated by works of Vetterli et al. (see \cite{hrlv,lv,rclv}), Aldroubi, Davis, and Krishtal first considered the case where $B$ is a convolution operator \cite{adk} (see also \cite{aadp}). The general Dynamical Sampling problem was tackled in the paper \cite{acmt}, recently followed by the successor \cite{accmp} (see also \cite{ap}). In \cite{acmt} the authors completely describe the finite-dimensional situation and characterize the frame property of $(A^ke)_{k\in\N}$ in the case where $A$ is a selfadjoint operator with eigenvectors forming an orthonormal basis of $\calH$. The paper \cite{accmp} deals with normal operators $A$ and provides several necessary and sufficient conditions under which $(A^ke_i)_{i\in I,\,k=0,\ldots,K_i}$ is minimal or complete. The conditions are formulated in terms of a certain decomposition of the normal operator $A$, taking into account the spectral multiplicity of $A$.

The present paper is motivated by the following simple observation: in order that $(A^ke_i)_{i\in I,\,k=0,\ldots,K_i}$ be a frame for $\calH$ it is necessary that each of the systems $(A^ke_i)_{k=0,\ldots,K_i}$, $i\in I$, is a Bessel sequence. As this is only interesting for those $i\in I$ with $K_i = \infty$, we consider systems of the form $(A^kx)_{k\in\N}$ and ask for which normal operators $A$ and which vectors $x$ the corresponding system is a Bessel sequence. It quickly turns out that the answer to this question solely depends on the properties of the measure $\mu_x := \|E(\cdot)x\|^2$, where $E$ is the spectral measure of the normal operator $A$: it is shown in Lemma \ref{l:translate} that $(A^kx)_{k\in\N}$ is a Bessel sequence if and only if the sequence $(z^k)_{k\in\N}$ of monomials is a Bessel sequence in $L^2(\mu_x)$.

In the case where $A$ is unitary we find a comparatively simple and explicit characterization (see Theorem \ref{t:unitary}): The system $(A^kx)_{k\in\N}$ is a Bessel sequence if and only if the measure $\mu_x$ is Lipschitz continuous (see Definition \ref{d:lipschitz}) with respect to the arc length measure on the unit circle $\T$. For selfadjoint operators $A$ we provide two characterizations in Theorem \ref{t:sa_charac} one of which is even more simple. It states that $(A^kx)_{k\in\N}$ is a Bessel sequence if and only if $\<A^kx,x\> = O(k^{-1})$ as $k\to\infty$. In the proof we make use of results from the theory of Hankel matrices. In the case of general normal operators $A$ it is necessary for $(A^kx)_{k\in\N}$ to be a Bessel sequence that the support of the measure $\mu_x$ lies in the closed unit disc $\ol\D$ (see Lemma \ref{l:necessary}). The Bessel sequence property of $(A^kx)_{k\in\N}$ is now highly dependent on the behaviour of the measure $\mu_x$ close to the unit circle line. In fact, we prove in Theorem \ref{t:hardy} that $(A^kx)_{k\in\N}$ is a Bessel sequence if and only if the support of $\mu_x$ lies in $\ol\D$, $\mu_x|\T$ is Lipschitz continuous with respect to the arc length measure and $\mu_x|\D$ is a Carleson measure. This theorem also provides a characterization in terms of resolvent growth. In the end of the paper we return to the Dynamical Sampling problem in connection with the diffusion operator in the heat equation and prove the negative result that in this case the resulting sequence is not a Bessel sequence.

The paper is organized as follows. In Section 2 we fix our notation and recall some notions that we will use. In Section 3 we present characterizations for $(A^kx)_{k\in\N}$ to be a Bessel sequence in the case of unitary and selfadjoint operators $A$. After that we characterize the Bessel sequence property of $(A^kx)_{k\in\N}$ for general normal operators $A$ in Section 4 and provide some sufficient conditions in Section 5. In Section 6 we show that Dynamical Sampling in a certain framework connected to the heat equation is not possible when the function is sampled at an infinite amount of times. The last section, Section 7, is an appendix containing auxiliary results from measure theory and operator theory which are used in the preceding sections.

\section{Notation, Preliminaries, and Setting}
By $\N$ we denote the natural numbers {\it including zero}, whereas $\N^* := \N\setminus\{0\}$. The Borel $\sigma$-algebra on $\C$ is denoted by $\frakB$. If $\Delta_0\in\frakB$ we set $\frakB(\Delta_0) := \{\Delta\in\frakB : \Delta\subset\Delta_0\}$. By $\T$ we denote the unit circle and by $\D$ the open unit disc in $\C$. Recall that the scalar product on $L^2(\T)$ is given by $\<f,g\> = \int_\T f(z)\ol{g(z)}\,d|z|$, where the arc length measure $\arc := |\cdot|$ is such that $\int_\T d|z| = 1$. By $\Leb_n$ we denote the Lebesgue measure on $\R^n$. For $z\in\C$ and $k\in\Z$ we set
\begin{equation}\label{e:e_k}
e_k(z) := 
\begin{cases}
z^k&\text{if }k\ge 0,\\
\ol z^{-k}&\text{if }k < 0.
\end{cases}
\end{equation}
Hence, for $z\in\T$ we have $e_k(z) = z^k$ and for $z\in\R$, $e_k(z) = z^{|k|}$, $k\in\Z$. In the following, we shall restrict the function $e_k$ on $\T$ and $\D$, respectively. If no confusion can arise, we write $e_k$ instead of $e_k|\T$ or $e_k|\D$. For example, the system $(e_k)_{k\in\Z}\subset L^2(\T)$ is an orthonormal basis of $L^2(\T)$. Therefore, each $f\in L^2(\T)$ has a representation $f = \sum_{k\in\Z}c_ke_k$ with $(c_k)_{k\in\Z}\in\ell^2(\Z)$.

One of the several equivalent definitions of the {\it Hardy space} $H^2(\D)$ is as follows:
\begin{align*}
H^2(\D) := \left\{f : \D\to\C\,\left|\,f(z) = \sum_{k=0}^\infty c_kz^k,\;(c_k)_{k\in\N}\in\ell^2(\N)\right\}\right..
\end{align*}
It is well known that if $f\in H^2(\D)$, $f(z) = \sum_{k=0}^\infty c_kz^k$, then $\lim_{r\upto 1}f(rz)$ exists for almost every $z\in\T$ and that the limit function $\tilde f$ (where we set $\tilde f(z) = 0$ if $\lim_{r\upto 1}f(rz)$ does not exist) is in the $L^2(\T)$-equivalence class $\sum_{k=0}^\infty c_ke_k$. By a famous theorem of Carleson (see \cite{ca}) the series $\sum_{k=0}^\infty c_kz^k$ converges for a.e.\ $z\in\T$. Consequently, we have that $\tilde f(z) = \sum_{k=0}^\infty c_kz^k$ for a.e.\ $z\in\T$. Equipped with the scalar product
$$
\<f,g\>_{H^2(\D)} := \big\<\tilde f,\tilde g\big\>_{L^2(\T)},\qquad f,g\in H^2(\D),
$$
the space $H^2(\D)$ becomes a Hilbert space. In the following we shall also consider the space $H^2(\ol\D)$ consisting of the functions $f : \ol\D\to\C$ such that $f|\D\in H^2(\D)$ and $f|\T = \widetilde{f|\D}$. For $f,g\in H^2(\ol\D)$ we set
$$
\<f,g\>_{H^2(\ol\D)} := \<f|\D,g|\D\>_{H^2(\D)}.
$$
Then also $H^2(\ol\D)$ is a Hilbert space.

Recall that a sequence $(y_k)_{k\in\N}$ of vectors in a Hilbert space $\calH$ is called a {\it Bessel sequence} if there exists $C > 0$ such that
$$
\sum_{k\in\N}|\<y,y_k\>|^2\,\le\,C\|y\|^2\quad\text{for all }y\in\calH.
$$
A constant $C > 0$ as above is called a {\it Bessel bound} for the Bessel sequence $(y_k)_{k\in\N}$. The following well known proposition provides necessary and sufficient conditions for $(y_k)_{k\in\N}$ to be a Bessel sequence.

\begin{prop}[\cite{c}]\label{p:charac_bessel}
Let $(y_k)_{k\in\N}$ be a sequence of vectors in $\calH$ and $C > 0$. Then the following statements are equivalent.
\begin{enumerate}
\item[{\rm (i)}]   $(y_k)_{k\in\N}$ is a Bessel sequence in $\calH$ \braces{with Bessel bound $C$}.
\item[{\rm (ii)}]  $\sum_{k\in\N}c_ky_k$ converges in $\calH$ for each $(c_k)_{k\in\N}\in\ell^2(\N)$ {\rm (}and
\begin{equation}\label{e:synthesis}
\left\|\sum_{k\in\N}c_ky_k\right\|\,\le\,C^{1/2}\|c\|_2\quad\text{for all }c = (c_k)_{k\in\N}\in\ell^2(\N).)
\end{equation}
\item[{\rm (iii)}] $\sum_{k\in\N}c_ky_k$ converges unconditionally in $\calH$ for each $(c_k)_{k\in\N}\in\ell^2(\N)$ \braces{and \eqref{e:synthesis} holds}.
\item[{\rm (iv)}]  The infinite {\em Gram matrix} $(\<y_k,y_j\>)_{k,j\in\N}$ defines a bounded linear operator on $\ell^2(\N)$ \braces{with operator norm at most $C$}.
\end{enumerate}
\end{prop}

Throughout this paper, let $A$ be a bounded normal operator in a Hilbert space $\calH$. We consider sequences of the form $(A^kx)_{k\in\N}$ with $x\in\calH$ and ask the question as to whether this sequence is a Bessel sequence. With regard to this problem, the spectrum and the spectral measure of $A$ play an important role. We denote them by $\sigma(A)$ and $E$, respectively. If $x,y\in\calH$, by $\mu_{xy}$ we denote the complex measure
$$
\mu_{xy}(\Delta) := \<E(\Delta)x,y\>,\quad\Delta\in\frakB.
$$
It is clear that from the knowledge of the finite positive Borel measure\footnote{For us, a Borel measure is a measure defined on the Borel $\sigma$-algebra of a topological space.} $\mu_x := \mu_{xx}$ we cannot recover $A$ and/or $x$. However, in our studies it will turn out that from this knowledge alone we can decide on whether $(A^kx)_{k\in\N}$ is a Bessel sequence or not. Hence, the answer to the question is intimately connected with the properties of the measure $\mu_x$. The following lemma already hints at this connection. Recall that the support of a measure is the set of all points for which every open neighborhood has positive measure.

\begin{lem}\label{l:necessary}
Let $A$ be a bounded normal operator in $\calH$ and $x\in\calH$ such that $(A^kx)_{k\in\N}$ is a Bessel sequence. Then
$$
x\in E(\ol\D)\calH\quad (\text{i.e., }\supp(\mu_x)\,\subset\,\ol\D).
$$
\end{lem}
\begin{proof}
For $r > 0$ put $\D_r := \{z\in\C : |z| < r\}$ and $x_r := E(\C\setminus\D_r)x$. Let $C > 0$ be the Bessel bound of $(A^kx)_{k\in\N}$. For arbitrary $r > 1$ and $n\in\N$ we have
$$
\left\|A^nx_r\right\|^4\,\le\,\sum_{k=0}^\infty\left|\left\<A^nx_r,A^{k+n}x_r\right\>\right|^2 = \sum_{k=0}^\infty\left|\left\<A^nx_r,A^{k+n}x\right\>\right|^2\,\le\,C\left\|A^nx_r\right\|^2,
$$
and hence $C\ge \|A^nx_r\|^2$. But
$$
\left\|A^nx_r\right\|^2 = \int_{\C\setminus\D_r}|z|^{2n}d\mu_{x}\,\ge\,r^{2n}\mu_{x}(\C\setminus\D_r) = r^{2n}\|x_r\|^2.
$$
Letting $n\to\infty$ we conclude that $x_r = 0$ for each $r > 1$. This proves the claim.
\end{proof}

For $x\in\calH$ we shall also set
$$
\calH_x := \ol\linspan\left\{E(\Delta)x : \Delta\in\frakB\right\}.
$$
The orthogonal projection in $\calH$ onto $\calH_x$ will be denoted by $P_x$. For statements on the interplay between $\calH_x$ and the measures $\mu_{xy}$ and $\mu_x$ we refer the reader to Subsection \ref{ss:aux_op}.

\section{Unitary and Selfadjoint Operators}
In this section we shall consider special cases of normal operators, namely the unitary and selfadjoint ones. In these cases our characterizations for $(A^kx)_{k\in\N}$ to be a Bessel sequence will be especially simple and explicit. For the formulation of our characterization for unitary operators in Theorem \ref{t:unitary} below we introduce the notion of Lipschitz continuity of set functions.

\begin{defn}\label{d:lipschitz}
Let $\nu,\mu : \Sigma\to\R$ be non-negative set functions on a measurable space $(\Omega,\Sigma)$. We say that $\nu$ is {\em Lipschitz continuous with respect to $\mu$} {\rm (}or {\em $\mu$-Lipschitz continuous}{\rm )} with constant $C > 0$ if for all $\Delta\in\Sigma$ we have
$$
\nu(\Delta)\,\le\,C\mu(\Delta).
$$
\end{defn}

Note that for measures $\nu$ and $\mu$ (just as for functions) Lipschitz continuity is a stronger notion than absolute continuity. In fact, Lemma \ref{l:liphoelder} shows that $\nu$ is $\mu$-Lipschitz continuous if and only if $\nu\ll\mu$ and $d\nu/d\mu\in L^\infty(\mu)$.

\subsection{Unitary Operators}
In the proof of the following theorem we make use of the auxiliary results in Section \ref{s:aux}. Recall that $\arc = |\cdot|$ denotes the arc length measure on $\T$.

\begin{thm}\label{t:unitary}
Let $A$ be a unitary operator in $\calH$, $x\in\calH$, and $C > 0$. Then the following statements are equivalent.
\begin{enumerate}
\item[{\rm (i)}]   The sequence $(A^kx)_{k\in\N}$ is a Bessel sequence with Bessel bound $C$.
\item[{\rm (ii)}]  The sequence $(A^kx)_{k\in\Z}$ is a Bessel sequence with Bessel bound $C$.
\item[{\rm (iii)}] $\mu_x$ is Lipschitz continuous with respect to $\arc$ with constant $C$.
\item[{\rm (iv)}]  For any $y\in\calH$ the measure $|\mu_{xy}|$ is Lipschitz continuous with respect to $\sqrt{\mu_y\arc}$ with constant $C^{1/2}$.
\item[{\rm (v)}] $\|(A - \la)^{-1}x\|^2\le C\big|1-|\la|^2\big|^{-1}$ for $|\la|\neq 1$.
\end{enumerate}
\end{thm}
\begin{proof}
(i)$\Sra$(ii). Let $y\in\calH$. Then for each $m\in\N$ we have
$$
\sum_{k=-m}^\infty\left|\left\<y,A^kx\right\>\right|^2 = \sum_{k=0}^\infty\left|\left\<A^my,A^kx\right\>\right|^2\,\le\,C\|A^my\|^2 = C\|y\|^2.
$$
(ii)$\Sra$(iii). For each $\Delta\in\frakB(\T)$ we have
\begin{equation}\label{e:ff}
\sum_{k\in\Z}\left|\int_\Delta z^{-k}\,d\mu_{x}\right|^2 = \sum_{k\in\Z}\left|\left\<E(\Delta)x,A^kx\right\>\right|^2\,\le\,C\|E(\Delta)x\|^2 = C\mu_x(\Delta).
\end{equation}
This shows that the sequence $(\int_\Delta z^{-k}\,d\mu_{x})_{k\in\Z}$ is an element of $\ell^2(\Z)$ for each $\Delta\in\frakB(\T)$. Now, define $f := \sum_{k\in\Z}(\int_\T z^{-k}\,d\mu_{x})e_k\in L^2(\T)$ and $d\nu := f\,d|z|$. Then for each $k\in\Z$,
$$
\int_\T z^{-k}\,d\mu_{x} = \<f,e_k\> = \int_\T f(z)z^{-k}\,d|z| = \int_\T z^{-k}\,d\nu.
$$
By the Stone-Weierstra\ss\ Theorem, $\linspan\{z^k : k\in\Z\}$ is dense in $C(\T)$. Hence, the functional $\vphi\in C(\T)'$, defined by $\vphi(g) := \int_\T g\,d(\mu_{x} - \nu)$, vanishes identically. Now, it is a consequence of Riesz' representation theorem for $C(\T)'$ that $\mu_{x} = \nu$ (see also \cite[p.\ 36]{k}). Therefore, $d\mu_{x}/d|z| = f$. For $\Delta\in\frakB(\T)$ we have
$$
\left\<\chi_\Delta f,e_k\right\> = \int_\Delta f(z)z^{-k}\,d|z| = \int_\Delta z^{-k}\,d\mu_{x}.
$$
Thus, it follows from \eqref{e:ff} that $\|f|\Delta\|_2^2\le C\mu_x(\Delta)$. Hence, Lemma \ref{l:lipschitz} implies that $\mu_x$ is Lipschitz continuous with respect to $\sqrt{\mu_x\arc}$ with constant $C^{1/2}$, which is (iii).

(iii)$\Sra$(iv). By Lemma \ref{l:mu_yx}, for any $y\in\calH$ we have that $|\mu_{yx}|(\Delta)\le\sqrt{\mu_y(\Delta)\mu_x(\Delta)}$, $\Delta\in\frakB$. Thus, if $\mu_x(\Delta)\le C|\Delta|$ for each $\Delta\in\frakB(\T)$ then $|\mu_{yx}|(\Delta)\le C^{1/2}\sqrt{\mu_y(\Delta)|\Delta|}$.

(iv)$\Sra$(i). Let $y\in\calH$. By Lemma \ref{l:lipschitz} we have that $f := d\mu_{yx}/d|z|\in L^2(\T)$ and $\|f\|_2^2\le C\mu_y(\T) = C\|y\|^2$. This shows $\<f,e_k\> = \int_\T f(z)z^{-k}\,d|z| = \int_\T z^{-k}\,d\mu_{yx}$ for $k\in\Z$ and thus
$$
\sum_{k\in\Z}\left|\left\<y,A^kx\right\>\right|^2 = \sum_{k\in\Z}\left|\int_\T z^{-k}d\mu_{yx}\right|^2 = \sum_{k\in\Z}|\<f,e_k\>|^2 = \|f\|_2^2\,\le\,C\|y\|^2.
$$
(iii)$\Sra$(v). In the rest of the proof we make use of the Poisson kernel $P$ and the Poisson integral $\calP[\cdot]$ (see page \pageref{page:poisson}). By Lemma \ref{l:liphoelder}, we have that $\mu_x = f\,d|z|$, where $f\in L^\infty(\T)$ with $\|f\|_\infty\le C$. Hence, for $\la\in\D$,
\begin{align*}
\|(A - \la)^{-1}x\|^2
&= \int_\T |z-\la|^{-2}\,d\mu_x(z)\,\le\, C\int_\T |z-\la|^{-2}\,d|z|\\
&= \frac C{1-|\la|^2}\int_\T\frac{1-|\la|^2}{|z-\la|^2}\,d|z| = \frac C{1-|\la|^2}\calP[\arc](\la) = \frac C{1-|\la|^2}.
\end{align*}
Making use of $P(\ol w^{-1},z) = -P(w,z)$, one similarly shows that $\|(A - \la)^{-1}x\|\le C(|\la|^2-1)^{-1}$ for $\la\in\C\setminus\ol\D$.

(v)$\Sra$(iii). Let $\Delta$ be an open arc in $\T$. For $r\in (0,1)$ we have
\begin{align*}
\int_\Delta\calP[\mu_x](rz)\,d|z|
&= (1-r^2)\int_\Delta\int_\T\frac 1 {|\la- rz|^2}\,d\mu_x(\la)\,d|z|.
\end{align*}
Define the continuous function $g_z:\T\to\C$ by $g_z(\la) := (\la-rz)^{-1}$, $z,\la\in\T$. Then $\int|g_z|^2\,d\mu_x = \|\int g_z\,dEx\|^2 = \|g_z(A)x\|^2 = \|(A - rz)^{-1}x\|^2$. Hence,
\begin{align*}
\int_\Delta\calP[\mu_x](rz)\,d|z|
&= (1-r^2)\int_\Delta\|(A - rz)^{-1}x\|^2\,d|z|\,\le\,C|\Delta|,
\end{align*}
where the last inequality follows from (v). Using Stieltjes' inversion formula (Lemma \ref{l:stieltjes}) we conclude that $\mu_x(\Delta)\le C|\Delta|$ for every open arc in $\T$. Thus, the latter relation holds for every (relatively) open set $\Delta\subset\T$. Now, (iii) follows from the fact that $\mu_x$ is a regular measure (see, e.g., \cite[Theorem 2.18]{r}).
\end{proof}

\begin{rem}
(a) Note that (iii) in Theorem \ref{t:unitary} implies that the function $f = d\mu_x/d|z|\in L^2(\T)$, defined in the proof of (ii)$\Sra$(iii), is in fact an element of $L^\infty(\T)$ with $\|f\|_\infty\le C$, see Lemma \ref{l:liphoelder}.

\medskip\noindent
(b) Alternatively, we can prove the equivalence (i)$\Slra$(iii) in Theorem \ref{t:unitary} with the help of Toeplitz matrix theory. Indeed, by Proposition \ref{p:charac_bessel} $(A^kx)_{k\in\N}$ is a Bessel sequence with bound $C$ if and only if its Gram matrix $G = (\<A^kx,A^jx\>)_{j,k=0}^\infty$ defines a bounded operator in $\ell^2(\N)$ with norm at most $C$. In the present case, $G = (\<x,A^{j-k}x\>)_{j,k=0}^\infty$ is a Toeplitz matrix, which, by a theorem of Toeplitz (see, e.g., \cite{bg}), defines a bounded operator on $\ell^2(\N)$ with norm at most $C$ if and only if there exists some $f\in L^\infty(\T)$, $\|f\|_\infty\le C$, such that $\<x,A^kx\> =  \hat f(k)$ for all $k\in\Z$, i.e., $\int_\T z^{-k}\,d\mu_x = \int_\T f(z)z^{-k}\,d|z|$. By \cite[p.\ 36]{k} this is equivalent to $d\mu_x = f\,d|z|$ with $f\in L^\infty(\T)$, and thus to (iii), see Lemma \ref{l:liphoelder}.

\medskip\noindent
(c) As seen above, if $A$ is unitary, $(A^kx)_{k\in\N}$ is a Bessel sequence if and only if its Gram matrix $G$ has the form $T_f$, where $f\in L^\infty(\T)$, $f\ge 0$, is the symbol of the Toeplitz matrix $T_f$. By \cite{d}, the spectrum of $G = T_f$ is given by $[\essinf f,\esssup f]$. In particular, the Gram operator of $(A^kx)_{k\in\N}$ is either invertible or its range is not closed. In turn, $(A^kx)_{k\in\N}$ is a frame sequence (i.e., a frame for its closed linear span) if and only  if it is a Riesz basic sequence. This holds if and only if $f$ is essentially bounded below by a positive constant, or, equivalently, if and only if $\mu_x$ and $\arc$ are Lipschitz equivalent (i.e., there exist $c,C > 0$ such that $c|\Delta|\le\mu_x(\Delta)\le C|\Delta|$ for all $\Delta\in\frakB(\T)$). For this it is necessary that the unitary operator $A$ be non-reductive (see \cite{wer}).
\end{rem}

\subsection{Selfadjoint Operators}
Let us now proceed with the case where the operator $A$ is selfadjoint. Assume for the moment that $(A^kx)_{k\in\N}$ is a Bessel sequence with bound $C > 0$. Let $\veps\in (0,1)$ and put $x_\veps := E((1-\veps,\infty))x$. Then
$$
C\|x_\veps\|^2\ge\sum_{k\in\N}\left|\left\<x_\veps,A^kx_\veps\right\>\right|^2 \ge\sum_{k\in\N}(1-\veps)^{2k}\|x_\veps\|^4 = \frac{\|x_\veps\|^4}{1 - (1-\veps)^2},
$$
and hence $\mu_x((1-\veps,\infty)) = \|x_\veps\|^2\le C\veps(2 - \veps)\le 2C\veps$. Simililarly, one shows that $\mu_x((-\infty,-1+\veps))\le 2C\veps$. That is, we have $\mu_x(|t| > 1-\veps) = O(\veps)$ as $\veps\to 0$. Our characterization result Theorem \ref{t:sa_charac} shows in particular that this necessary condition is in fact also sufficient for $(A^kx)_{k\in\N}$ to be a Bessel sequence. In the proof we observe that the problem is strongly connected with the boundedness of a certain Hankel matrix. Recall that a Hankel matrix is an infinite matrix of the form $H = (a_{k+j})_{j,k\in\N}$, where $(a_k)_{k\in\N}$ is a sequence of, in general, complex numbers. Nehari showed in \cite{n} that $H$ defines a bounded operator on $\ell^2(\N)$ if and only if there exists a bounded measurable function $\psi$ on $\T$ such that $a_k = \<\psi,e_k\>_{L^2(\T)}$ for $k\in\N$. However, the next theorem is much better suited for our purposes. It was proved in \cite[Theorem 3.1]{w} (see also \cite[Theorem 1.3]{y}).

\begin{thm}\label{t:hankel}
Let $\mu$ be a finite positive measure on $\R$ such that $\mu(|t|\ge 1) = 0$ and define
$$
q_k := \int_{(-1,1)} t^k\,d\mu(t),\quad k\in\N.
$$
Then the following statements are equivalent.
\begin{enumerate}
\item[{\rm (i)}]   The Hankel matrix $(q_{k+j})_{j,k\in\N}$ defines a bounded operator in $\ell^2(\N)$.
\item[{\rm (ii)}]  $\mu(|t| > 1-\veps) = O(\veps)$ as $\veps\to 0$.
\item[{\rm (iii)}] $q_k = O(k^{-1})$ as $k\to\infty$.
\end{enumerate}
\end{thm}

We are now ready to prove our main result for the case where $A$ is a selfadjoint operator.

\begin{thm}\label{t:sa_charac}
Let $A$ be a bounded selfadjoint operator in $\calH$ and $x\in\calH$. Then the following statements are equivalent.
\begin{enumerate}
\item[{\rm (i)}]   $(A^kx)_{k\in\N}$ is a Bessel sequence in $\calH$.
\item[{\rm (ii)}]  $\mu_x(|t| > 1-\veps) = O(\veps)$ as $\veps\to 0$.
\item[{\rm (iii)}] $\<A^kx,x\> = O(k^{-1})$ as $k\to\infty$.
\end{enumerate}
\end{thm}
\begin{proof}
Let us first see that all three conditions (i)--(iii) imply that $\mu_x(|t|\ge 1) = 0$. This is clear for (ii), but also for (i), since (ii) follows from (i) by the discussion above. If (iii) is satisfied, then for all $k\in\N$ we have
$$
\mu_x(|t|\ge 1)\,\le\,\int_{\{|t|\ge 1\}}t^{2k}\,d\mu_x\,\le\,\int_\R t^{2k}\,d\mu_x = \left\<A^{2k}x,x\right\>\le\frac C{2k}
$$
with some $C > 0$. Hence, $\mu_x(|t|\ge 1) = 0$. Therefore, we can a priori assume this so that $\mu_x$ satisfies the conditions in Theorem \ref{t:hankel}. By Proposition \ref{p:charac_bessel}, a sequence $(y_k)_{k\in\N}$ is a Bessel sequence if and only if its Gram matrix
$$
G := \left(\left\<y_k,y_j\right\>\right)_{j,k\in\N}
$$
defines a bounded operator in $\ell^2(\N)$. In the present case, we have $G = (q_{k+j})_{j,k\in\N}$, where
$$
q_k := \<A^kx,x\> = \int_{(-1,1)} t^k\,d\mu_x,\quad k\in\N.
$$
In other words, $G$ is a Hankel matrix of the type considered in Theorem \ref{t:hankel}. The assertion now follows directly from there.
\end{proof}

\begin{rem}
We would like to mention that by \cite[Theorem 1.2]{y} the necessary condition that $\mu_x(|t|\ge 1) = 0$ is equivalent to the fact that the quadratic form $(c_k)\mapsto \sum_k\sum_j q_{k+j}c_k\ol{c_j}$, defined on the sequences with only finitely many non-zero components, is closable on $\ell^2(\N)$.
\end{rem}

\begin{ex}
Let $A$ be the multiplication operator in $L^2(-1,1)$ with the free variable (i.e., $(Af)(t) = tf(t)$). Then $(e_k)_{k\in\N} = (A^k\one)_{k\in\N}$ is a Bessel sequence in $L^2(-1,1)$. Indeed, here $E(\Delta)f = \chi_\Delta f$, so that $\mu_\one = \Leb_1$. Hence, the assertion follows from Theorem \ref{t:sa_charac}. Since for $c = (c_k)_{k\in\N}\in\ell^2(\N)$ we have (using Hilbert's inequality (see, e.g., \cite{hlp}))
$$
\left\|\sum_{k=0}^\infty c_ke_k\right\|_{L^2(-1,1)}^2 = \int_{-1}^1\sum_{k=0}^\infty\sum_{j=0}^\infty c_k\ol{c_j}t^{k+j}\,dt\,\le\,2\sum_{k=0}^\infty\sum_{j=0}^\infty\frac{|c_kc_j|}{k+j+1}\,\le\,2\pi\|c\|_2^2,
$$
a Bessel bound is given by $2\pi$. It is well known that $(e_k)_{k\in\N}$ is also complete in $L^2(-1,1)$. However, since $\sigma(A) = [-1,1]$, it follows from, e.g., \cite[Corollary 1]{ap} that $(e_k)_{k\in\N}$ is not a frame for $L^2(-1,1)$.
\end{ex}

\section{General Normal Operators}
We shall now consider general bounded normal operators $A$ in $\calH$. The next lemma transfers the original problem to an analogue problem in an $L^2$-space. Notice the equivalence of (i) and (ii) which is not obvious at first sight. Recall that $e_k(z) = z^k$ for $k\in\N$.

\begin{lem}\label{l:translate}
Let $A$ be a bounded normal operator in $\calH$, $x\in\calH$, and $C > 0$. Then the following statements are equivalent.
\begin{enumerate}
\item[{\rm (i)}]   The sequence $(A^kx)_{k\in\N}$ is a Bessel sequence in $\calH$ with Bessel bound $C$.
\item[{\rm (ii)}]  The sequence $((A^*)^kx)_{k\in\N}$ is a Bessel sequence in $\calH$ with Bessel bound $C$.
\item[{\rm (iii)}] The sequence $(e_k)_{k\in\N}$ is a Bessel sequence in $L^2(\mu_x)$ with Bessel bound $C$.
\end{enumerate}
\end{lem}
\begin{proof}
Let $y\in\calH$. Then, by Lemma \ref{l:mu_yx}, $d\mu_{yx} = f_y\,d\mu_x$ with some $f_y\in L^2(\mu_x)$. We have
\begin{equation}\label{e:translate}
\left\<y,A^kx\right\> = \int \ol z^k\,d\mu_{yx} = \int f_y(z)\ol{z^k}\,d\mu_x = \<f_y,e_k\>_{L^2(\mu_x)}.
\end{equation}
And, similarly,
\begin{equation}\label{e:translate*}
\left\<y,(A^*)^kx\right\> = \int z^k\,d\mu_{yx} = \int f_y(z)z^k\,d\mu_x = \<e_k,\ol{f_y}\>_{L^2(\mu_x)}.
\end{equation}
Recall that the operator $L_x : \calH_x\to L^2(\mu_x)$ mapping $y\in\calH_x$ to $f_y$ is an isometric isomorphism by Lemma \ref{l:mu_yx}.

(i)$\Sra$(iii). Let $f\in L^2(\mu_x)$. Then there exists $y\in\calH_x$ such that $f = f_y$. By \eqref{e:translate} we have
$$
\sum_{k=0}^\infty\left|\<f,e_k\>_{L^2(\mu_x)}\right|^2 = \sum_{k=0}^\infty\left|\left\<y,A^kx\right\>\right|^2\le C\|y\|^2 = C\|f\|_{L^2(\mu_x)}^2.
$$
(iii)$\Sra$(i). Let $y\in\calH$. Since for $u\in\calH_x^\perp$ we have $\mu_{ux} = 0$ and thus $f_u = 0$, it follows that $\|f_y\| = \|L_xP_xy\| = \|P_xy\|\le\|y\|$. Hence,
$$
\sum_{k=0}^\infty\left|\left\<y,A^kx\right\>\right|^2 = \sum_{k=0}^\infty\left|\left\<f_y,e_k\right\>_{L^2(\mu_x)}\right|^2\,\le\,C\|f_y\|_{L^2(\mu_x)}^2\,\le\,C\|y\|^2.
$$
(i)$\Slra$(ii). Assume that $(A^kx)_{k\in\N}$ is a Bessel sequence with bound $C$. Let $y\in\calH$. Then, by \eqref{e:translate*} and since $(e_k)_{k\in\N}$ is a Bessel sequence in $L^2(\mu_x)$, we have
$$
\sum_{k=0}^\infty\left|\left\<y,(A^*)^kx\right\>\right|^2 = \sum_{k=0}^\infty\left|\left\<\ol{f_y},e_k\right\>\right|^2\,\le\,C\|\ol{f_y}\|_{L^2(\mu_x)}^2\le C\|y\|^2.
$$
This proves (ii). The implication (ii)$\Sra$(i) follows by interchanging the roles of $A$ and $A^*$.
\end{proof}

\begin{rem}
Note that we cannot replace $(e_k)_{k\in\N}$ in Lemma \ref{l:translate} by $(e_k)_{k\in\Z}$ (see \eqref{e:e_k}) as in the unitary case. This is especially revealed in the selfadjoint case where $e_k = e_{-k}$.
\end{rem}

Let $\mu$ be a finite Borel measure on $\C$. By $\scrL^2(\mu)$ we denote the space of all functions (not equivalence classes!) $f$ such that $\|f\|_{L^2(\mu)}^2 = \int |f|^2\,d\mu < \infty$. By saying that a normed function space $(M,\|\cdot\|_M)$ is continuously embedded in $L^2(\mu)$ we mean that $M$ is a subspace of $\scrL^2(\mu)$ and that there exists $C > 0$ such that $\|\psi\|_{L^2(\mu)}^2\le C\|\psi\|_M^2$ for all $\psi\in M$. We write $M\hookrightarrow L^2(\mu)$. The constant $C$ will be called the {\it continuity bound} of the continuous embedding.

\begin{thm}\label{t:hardy1}
Let $A$ be a bounded normal operator in $\calH$, $x\in\calH$, and $C > 0$. Then the sequence $(A^kx)_{k\in\N}$ is a Bessel sequence in $\calH$ with Bessel bound $C$ if and only if the following conditions hold.
\begin{enumerate}
\item[{\rm (i)}]  $x\in E(\ol\D)\calH$ {\rm (}i.e., $\supp(\mu_x)\subset\ol\D${\rm )}.
\item[{\rm (ii)}] $H^2(\ol\D)\hookrightarrow L^2(\mu_x)$ with continuity bound $C$.
\end{enumerate}
\end{thm}
\begin{proof}
Assume that $(A^kx)_{k\in\N}$ is a Bessel sequence with Bessel bound $C$. Then (i) holds due to Lemma \ref{l:necessary}. We also mention the following: since $A_\T := A|E(\T)\calH$ is a unitary operator in $E(\T)\calH$ and $(A_\T^kE(\T)x)_{k\in\N}$ is easily seen to be a Bessel sequence in $E(\T)\calH$, it is a consequence of Theorem \ref{t:unitary} that $\mu_x|\T$ is Lipschitz continuous with respect to $\arc$. In particular, $\mu_x|\T\ll\arc$. By Lemma \ref{l:translate}, $(e_k)_{k\in\N}$ is a Bessel sequence in $L^2(\mu_x)$ with Bessel bound $C$. Hence, for each $c = (c_k)_{k\in\N}\in\ell^2(\N)$, the series $\sum_{k=0}^\infty c_ke_k$ converges in $L^2(\mu_x)$ and $\|\sum_{k=0}^\infty c_ke_k\|_{L^2(\mu_x)}^2\le C\|c\|_2^2$. Now, let $f\in H^2(\ol\D)$. Then there exists $c = (c_k)_{k\in\N}\in\ell^2(\N)$ such that $f(z) = \sum_{k=0}^\infty c_kz^k$ for $z\in\D$ and for $\arc$-a.e.\ $z\in\T$. Put $f_n := \sum_{k=0}^nc_ke_k$. Then $f_n(z)\to f(z)$ for each $z\in\D$ and $f_n(z)\to f(z)$ as $n\to\infty$ for $\arc$-a.e. $z\in\T$. Since $\mu_x|\T\ll\arc$, we conclude that $f_n(z)\to f(z)$ as $n\to\infty$ for $\mu_x$-a.e. $z\in\ol\D$. Hence, (any representative of) the $L^2(\mu_x)$-limit $\sum_{k=0}^\infty c_ke_k$ of $(f_n)$ coincides with $f$ $\mu_x$-a.e. on $\ol\D$, that is, $f\in L^2(\mu_x)$ and
$$
\|f\|_{L^2(\mu_x)}^2 = \left\|\sum_{k=0}^\infty c_ke_k\right\|_{L^2(\mu_x)}^2\le\, C\|c\|_2^2 = C\|f\|_{H^2(\ol\D)}^2.
$$
This shows that $H^2(\ol\D)$ is continuously embedded in $L^2(\mu_x)$ with continuity bound $C$.

Conversely, assume that the conditions (i) and (ii) are satisfied. Let $c = (c_k)_{k\in\N}\in\ell^2(\N)$. Then
\begin{align*}
\left\|\sum_{k=n}^{n+m}c_ke_k\right\|_{L^2(\mu_x)}^2
&\le C\left\|\sum_{k=n}^{n+m}c_ke_k\right\|_{H^2(\ol\D)}^2
= C\sum_{k=n}^{n+m}|c_k|^2.
\end{align*}
This shows that the series $\sum_{k=0}^\infty c_ke_k$ converges in $L^2(\mu_x)$ and that $\|\sum_{k=0}^\infty c_ke_k\|_{L^2(\mu_x)}^2\le C\|c\|_2^2$. Hence, $(e_k)_{k\in\N}$ is a Bessel sequence in $L^2(\mu_x)$ with Bessel bound $C$. The claim now follows from Lemma \ref{l:translate}.
\end{proof}

Recall that we were aiming to find characterizations of the Bessel property of $(A^kx)_{k\in\N}$ solely in terms of the measure $\mu_x$. Although items (i) and (ii) in Theorem \ref{t:hardy1} are such conditions, condition (ii) is somewhat implicit. In order to state an equivalent, but more explicit condition, we shall utilize the following theorem which is known under the term {\it Carleson embedding theorem} (see, e.g., \cite{ga} or \cite{ni}). By $B_r(z)$ we denote the circle of radius $r > 0$ and center $z\in\C$.

\begin{thm}\label{t:carleson_embedding}
Let $\mu$ be a finite Borel measure on the unit disc $\D$. Then the following statements are equivalent.
\begin{enumerate}
\item[{\rm (i)}]   $H^2(\D)$ is continuously embedded in $L^2(\mu)$.
\item[{\rm (ii)}]  There exists $C_1 > 0$ such that for all $r > 0$ and $z\in\T$ we have
$$
\mu(\D\cap B_r(z))\le C_1r.
$$
\item[{\rm (iii)}] We have
$$
\sup_{z\in\D}\int_\D\frac{1-|z|^2}{|1-\ol zw|^2}\,d\mu(w) < \infty.
$$
\end{enumerate}
\end{thm}

A measure $\mu$ satisfying one of the conditions in Theorem \ref{t:carleson_embedding} is called a {\it Carleson measure}. The following theorem characterizes the Bessel property of $(A^kx)_{k\in\N}$ in terms of both the measure $\mu_x$ and the expression $\|(A_{\ol\D}-\la)^{-1}x\|$, where $|\la| > 1$ and $A_{\ol\D}$ is the restriction of the normal operator $A$ to its spectral subspace $E(\ol\D)\calH$.

\begin{thm}\label{t:hardy}
Let $A$ be a bounded normal operator in $\calH$ and $x\in E(\ol\D)\calH$. Then the following statements are equivalent.
\begin{enumerate}
\item[{\rm (i)}]   The sequence $(A^kx)_{k\in\N}$ is a Bessel sequence in $\calH$.
\item[{\rm (ii)}]  There exists $C_1 > 0$ such that for all $r>0$ and $z\in\T$ we have
$$
\mu_x(\ol\D\cap B_r(z))\le C_1r.
$$
\item[{\rm (iii)}] $\|(A_{\ol\D} - \la)^{-1}x\|^2 = O((|\la|^2-1)^{-1})$ as $|\la|\downto 1$.
\end{enumerate}
\end{thm}
\begin{proof}
(i)$\Slra$(ii). Assume that $(A^kx)_{k\in\N}$ is a Bessel sequence with Bessel bound $C$. As in the proof of Theorem \ref{t:hardy1} we infer that $\mu_x(\Delta)\le C|\Delta|$ for $\Delta\in\frakB(\T)$. As a consequence of Theorem \ref{t:hardy1}, $H^2(\D)$ is continuously embedded in $L^2(\mu_x)$. Thus, $\mu_x|\D$ is a Carleson measure which implies that there exists $C' > 0$ such that for all $r>0$ and $z\in\T$ we have\footnote{For the last estimate note that $\arc$ is normalized, i.e., $|\T| = 1$.}
$$
\mu_x(\ol\D\cap B_r(z))\,\le\,C'r + \mu_x(\T\cap B_r(z))\,\le\,C'r + C|\T\cap B_r(z)|\,\le\,(C' + C/2)r.
$$
Assume now that condition (ii) is satisfied. Then, in particular, $\mu_x|\T$ is Lipschitz continuous with respect to arc length measure. Therefore, due to Theorem \ref{t:unitary} and Lemma \ref{l:translate}, $(e_k|\T)_{k\in\N}$ is a Bessel sequence in $L^2(\mu_x)$. Moreover, Theorem \ref{t:carleson_embedding} implies that $H^2(\D)$ is continuously embedded in $L^2(\mu_x|\D)$. Equivalently, $(e_k|\D)_{k\in\N}$ is a Bessel sequence in $L^2(\mu_x)$. Hence, $(e_k)_{k\in\N}$ is a Bessel sequence in $L^2(\mu_x)$, and the claim follows from Lemma \ref{l:translate}.

(i)$\Slra$(iii). Put $A_\T := A|E(\T)\calH$, $A_\D := A|E(\D)\calH$, $x_\T := E(\T)x$, and $x_\D := E(\D)x$. Then $A_{\ol\D} = A_\T\oplus A_\D$ and (iii) holds if and only if the following two conditions are satisfied:
\begin{enumerate}
\item[(a)] $\|(A_{\T} - \la)^{-1}x_\T\|^2 = O((|\la|^2-1)^{-1})$ as $|\la|\downto 1$.
\item[(b)] $\|(A_{\D} - \la)^{-1}x_\D\|^2 = O((|\la|^2-1)^{-1})$ as $|\la|\downto 1$.
\end{enumerate}
From Theorem \ref{t:unitary} we know that (a) is equivalent to $(A_\T^kx_\T)_{k\in\N}$ being a Bessel sequence in $E(\T)\calH$. Hence, the equivalence (i)$\Slra$(iii) is proved if we can show that (b) holds if and only if $(A_\D^kx_\D)_{k\in\N}$ is a Bessel sequence in $E(\D)\calH$. For this, let $\la\in\C$, $|\la| > 1$, and put $z = \ol\la^{-1}$. Then
\begin{align*}
(|\la|^2 - 1)\|(A_\D - \la)^{-1}x_\D\|^2
&= \frac{1-|z|^2}{|z|^2}\|(\ol z^{-1} - A_\D)^{-1}x_\D\|^2\\
&= (1-|z|^2)\|\ol z^{-1}(\ol z^{-1} - A_\D)^{-1}x_\D\|^2\\
&= (1-|z|^2)\|(1 - \ol zA_\D)^{-1}x_\D\|^2\\
&= \int_\D \frac{1-|z|^2}{|1-\ol zw|^{2}}\,d\mu_x.
\end{align*}
Hence, (b) is satisfied if and only if (iii) in Theorem \ref{t:carleson_embedding} holds for the measure $\mu = \mu_x$. This is true if and only if $\mu_x|\D$ is a Carleson measure. Thus, the already proved equivalence (i)$\Slra$(ii) yields the claim.
\end{proof}

\begin{rem}
(a) The condition (ii) in Theorem \ref{t:hardy} can be equivalently split up into the following two conditions:
\begin{enumerate}
\item[(iia)] $\mu_x|\T$ is Lipschitz continuous with respect to arc length measure.
\item[(iib)] There exists $C_1 > 0$ such that for all $r>0$ and $z\in\T$ we have
$$
\mu_x(\D\cap B_r(z))\le C_1r.
$$
\end{enumerate}
(b) The equivalence (i)$\Slra$(ii) in Theorem \ref{t:sa_charac} follows from Theorem \ref{t:hardy}. However, we preferred to treat the selfadjoint case separately in order to demonstrate and emphasize the close connection to Hankel matrices. Although also (i)$\Slra$(iii)$\Slra$(v) in Theorem \ref{t:unitary} follows from Theorem \ref{t:hardy}, we advert to the fact that we have used these equivalences in the proofs of Theorems \ref{t:hardy1} and \ref{t:hardy}.

\smallskip\noindent
(c) Note that -- in contrast to Theorem \ref{t:hardy1} -- Theorem \ref{t:hardy} does not take account of the Bessel bound. On the other hand, as already mentioned, the conditions on the measure $\mu_x$ in Theorem \ref{t:hardy} are more explicit than those in Theorem \ref{t:hardy1}.
\end{rem}

We close this section with the following corollary providing a necessary condition on a system of the form $(A^kx_i)_{i\in I,\,k\in\N}$ to be a complete Bessel sequence. A similar statement can be found in \cite{accmp}.

\begin{cor}\label{c:morevecs}
Let $A$ be a bounded normal operator in $\calH$ and $(x_i)_{i\in I}\subset\calH$ be such that the system $(A^kx_i)_{i\in I,\,k\in\N}$ is a complete Bessel sequence in $\calH$ with Bessel bound $C > 0$. Then $\sigma(A)\subset{\ol\D}$ and $E|\T$ is absolutely continuous with respect to the arc length measure.
\end{cor}
\begin{proof}
Clearly, for each $i\in I$ the system $(A^kx_i)_{k\in\N}$ is a Bessel sequence. Therefore, $x_i\in E({\ol\D})\calH$ for each $i\in I$ by Lemma \ref{l:necessary}. But then $\ol\linspan\{A^kx_i : i\in I,\,k\in\N\}\subset E({\ol\D})\calH$. We conclude that $E({\ol\D})\calH = \calH$ and thus $\sigma(A)\subset{\ol\D}$. Let $\Delta\in\frakB(\T)$ with arc length measure zero. Then, by Theorem \ref{t:unitary}, $E(\Delta)x_i = 0$ for all $i\in I$. Hence $E(\Delta)\calH = E(\Delta)\ol\linspan\{A^kx_i : i\in I,\,k\in\N\} = \ol\linspan\{A^k E(\Delta)x_i : i\in I,\,k\in\N\} = \{0\}$, that is, $E(\Delta) = 0$.
\end{proof}

\begin{rem}
If $(A^kx)_{k\in\N}$ is complete but fails to be a Bessel sequence, it may well be that a subsequence $(A^{k_j}x)_{j\in\N}$ is a Bessel sequence. If, in addition, we have that $k_1 = 0$ and $\sum_j k_j^{-1} = \infty$, then $(A^{k_j}x)_{j\in\N}$ remains complete. This was proved in \cite{acmt} as a consequence of the M\"untz-Sz\'asz Theorem.
\end{rem}

\section{Sufficient Conditions}
In this section we provide conditions which imply that $(A^kx)_{k\in\N}$ is a Bessel sequence in $\calH$.

\begin{prop}\label{p:suff_bessel}
Let $A$ be a bounded normal operator in $\calH$ and $x\in E(\ol\D)\calH$. If the function $z\mapsto (1 - |z|^2)^{-1}$ lies in $L^1(\mu_x)$, then $(A^kx)_{k\in\N}$ is a Bessel sequence in $\calH$ with Bessel bound $\|(1 - |z|^2)^{-1}\|_{L^1(\mu_x)}$.
\end{prop}
\begin{proof}
Let $y\in\calH$. Then, since $|\mu_{yx}|(\Delta)\le\sqrt{\mu_x(\Delta)\mu_y(\Delta)}$ for any $\Delta\in\frakB$ (see Lemma \ref{l:mu_yx}), we have
\begin{align*}
\sum_{k=0}^\infty\left|\left\<y,A^kx\right\>\right|^2
&= \sum_{k=0}^\infty\left|\int_{\ol\D} \ol z^{k}\,d\mu_{yx}\right|^2\,\le\,\sum_{k=0}^\infty\left(\int_{\ol\D} |z|^{k}\,d|\mu_{yx}|\right)^2\\
&\le \sum_{k=0}^\infty\left(\int_{\ol\D} |z|^{2k}\,d\mu_x\right)\left(\int_{\ol\D}\,d\mu_y\right)\\
&= \left(\int_{\ol\D}\frac{1}{1-|z|^{2}}\,d\mu_x\right)\mu_y(\ol\D)\,\le\,\left\|\frac 1 {1 - |z|^2}\right\|_{L^1(\mu_x)}\|y\|^2.
\end{align*}
This proves the proposition.
\end{proof}

\begin{cor}\label{c:supp}
Let $A$ be a bounded normal operator in $\calH$ and $x\in\calH$. If $\supp(\mu_x)\subset\D$ then $(A^kx)_{k\in\N}$ is a Bessel sequence in $\calH$ with Bessel bound
$$
\frac{\|x\|^2}{1 - \nu^2},
$$
where $\nu = \sup\{|\la| : \la\in\supp(\mu_x)\}$.
\end{cor}

Recall that by $\Leb_n$ we denote the Lebesgue measure on $\R^n$. Moreover, let $\Lambda_\rho^r$ denote the open annulus in $\C$ with center zero and radii $\rho < r$. The following proposition is a direct consequence of Theorem \ref{t:hardy}.

\begin{prop}\label{p:suff_lip}
Let $A$ be a bounded normal operator in $\calH$ and $x\in E(\ol\D)\calH$ such that $\mu_x|\T$ is Lipschitz continuous with respect to the arc length measure. Assume furthermore that $\mu_x|\Lambda_{1-\veps}^1$ is Lipschitz continuous with respect to $\Leb_2$ for some $\veps\in (0,1)$. Then $(A^kx)_{k\in\N}$ is a Bessel sequence in $\calH$.
\end{prop}

\begin{rem}
Note that the sufficient conditions in Proposition \ref{p:suff_bessel} and Proposition \ref{p:suff_lip} cannot be satisfied at the same time if $d\mu_x/d\Leb_2\ge\delta > 0$ on some annulus $\Lambda_{1-\veps}^1$ since $z\mapsto(1 - |z|^2)^{-1}$ is not Lebesgue integrable over $\Lambda_{1-\veps}^1$.
\end{rem}

In the selfadjoint case, the condition in Proposition \ref{p:suff_lip} that $\mu_x$ be Lipschitz continuous with respect to $\Leb_2$ in an annulus $\Lambda_{1-\veps}^1$ means that the support of $\mu_x$ is (compactly) contained in $(-1,1)$. Therefore, the assumption in the next corollary (namely, that $\mu_x$ be Lipschitz continuous with respect to $\Leb_1$ outside an interval $[-1+\veps,1-\veps]$) is weaker. However, it still ensures that $(A^kx)_{k\in\N}$ is a Bessel sequence. Since it more or less directly follows from Theorem \ref{t:sa_charac}, we omit the proof.

\begin{prop}\label{p:sa_suff}
Let $A$ be selfadjoint and $x\in E((-1,1))\calH$. Assume that there exists $\veps > 0$ such that $\mu_x|\Delta_\veps$ is Lipschitz continuous with respect to $\Leb_1$, where $\Delta_\veps := \{t\in (-1,1) : |t| > 1-\veps\}$. Then $(A^kx)_{k\in\N}$ is a Bessel sequence in $\calH$.
\end{prop}

Example \ref{e:discrete} below shows that there are Bessel sequences $(A^kx)_{k\in\N}$ for which $\mu_x|\Delta_\veps$ is not even absolutely continuous with respect to Lebesgue measure for any $\veps\in (0,1)$.

\begin{ex}\label{e:discrete}
Let $\calH = \ell^2 := \ell^2(\N^*)$, $A = \diag(1 - \frac 1 n)_{n\ge 1}$, and $x = \sum_{n=1}^\infty 2^{-n}\delta_n$, where $\delta_n$ denotes the $n$-th standard basis vector of $\ell^2$. Then 
$\mu_x$ is the discrete measure with mass $2^{-2n}$ at $1 - 1/n$, $n\ge 1$. Thus,
$$
\int\frac 1 {1-t^2}\,d\mu_x = \sum_{n=1}^\infty 2^{-2n}\frac{1}{1 - (1-\frac 1 n)^2} = \sum_{n=1}^\infty 2^{-2n}\frac{n^2}{2n-1} < \infty,
$$
and $(A^kx)_{k\in\N}$ is a Bessel sequence by Proposition \ref{p:suff_bessel}.
\end{ex}

\section{Dynamical Sampling of Solutions to the Heat Equation}
In this section we make use of the Fourier transform which we define for functions $f\in L^2(\R)$ by
$$
\calF f := \hat f := \liim_{N\to\infty}\int_{-N}^N e^{2\pi ix\cdot}f(x)\,dx,
$$
where $\liim$ denotes the limit in $L^2(\R)$. We would like to sample bivariate functions $f(t,x)$ at several points in time and space of which we know that they obey the heat equation
$$
\frac{\partial f}{\partial t} - \frac{\partial^2 f}{\partial x^2} = 0.
$$
We assume that for each\footnote{It actually suffices to assume this only for $t=0$.} $t\in [0,\infty)$ the function $f(t,\cdot)$ lies in the space of band-limited functions $PW_{1/2}$, where
$$
PW_a := \left\{g : \R\to\C\,\left|\,g = \int_{-a}^{a}e^{2\pi ix\cdot}\phi(x)\,dx,\,\phi\in L^2(-a,a)\right.\right\}
$$
for $a > 0$. It is a well known fact that each $g\in\BL:=PW_{1/2}$ admits an extension to an entire function and that also $g'\in\BL$. In particular, $\BL\subset H^2(\R)$, where $H^2(\R)$ denotes the Sobolev space in $L^2(\R)$ of second regularity order. Moreover, if we equip $\BL$ with the $L^2$-norm, then we have $\BL\cong L^2(-1/2,1/2)$ via Fourier transform and restriction. In particular, $\BL$ is a Hilbert space\footnote{In fact, $\BL$ is a reproducing kernel Hilbert space with a $\sinc$ kernel.}.

Setting $u(t) := f(t,\cdot)$, the heat equation reduces to $\dot u(t) = Bu(t)$, where
$$
B : \BL\to\BL,\quad Bu := u'',\quad u\in\BL.
$$
As $\wh{u''}(\xi) = -\xi^2\wh u(\xi)$ for $u\in H^2(\R)$, we have indeed $Bu\in\BL$ for $u\in\BL$. If we regard the Fourier transform $\calF$ as an operator from $\BL$ to $L^2(-1/2,1/2)$, we have that
$$
(\calF B\calF^*\phi)(\xi) = -\xi^2\phi(\xi),\quad\phi\in L^2(-1/2,1/2),\;\xi\in (-1/2,1/2).
$$
Therefore (see, e.g., \cite[Theorem VII.5.1]{b}), $B$ is a bounded selfadjoint operator with $\sigma(B) = [-1/4,0]$.

By $\calS$ we denote our sampling scheme, i.e., $\calS\subset [0,\infty)\times\R$ and our samples are $\{f(t,x) : (t,x)\in\calS\}$. For simplicity, we assume that
$$
\calS\,\subset\,\{(k\delta,mi) : k\in\N,\,i\in\Z\},
$$
where $m\in\N^*$ and $\delta > 0$ are fixed. The information that $f$ obeys the heat equation is equivalent to $u(t) = e^{tB}u_0$, $t\ge 0$. Hence, it is our aim to recover $u_0 = u(0) = f(0,\cdot)$. We put $A := e^{\delta B}$, which is a positive definite selfadjoint operator in $\BL$ with $\sigma(A) = [e^{-\delta/4},1]$. Our samples are then of the form
$$
f(k\delta,mi) = \left(u(k\delta)\right)(mi) = \left(e^{k\delta B}u_0\right)(mi) = \left(A^ku_0\right)(mi) = \left\<A^ku_0,T^{mi}\sinc\right\>,
$$
where $T$ denotes the operator that translates a function by one and $\sinc(t) = \sin(\pi t)/(\pi t)$. Setting $\frake_i := T^i\sinc$ for $i\in\Z$, we have $f(k\delta,mi) = \<u_0,A^k\frake_{mi}\>$. Therefore, the function $f$ can be stably recovered from the samples if and only if $(A^k\frake_{mi})_{(k,i)\in\calD}$ is a frame for $\BL$, where $\calD := \{(k,i)\in\N\times\Z : (k\delta,mi)\in\calS\}$. For example, if we only sample at time $t=\kappa\delta$ and at all places $i\in\Z$, we have $(A^k\frake_{mi})_{(k,i)\in\calD} = (A^\kappa\frake_i)_{i\in\Z}$, which is a Riesz basis of $\BL$ since $(\frake_i)_{i\in\Z}$ is an orthonormal basis of $\BL$ and $A$ is invertible.

But let us try to decide whether $(A^k\frake_i)_{k\in\N}$ is a Bessel sequence for some $i\in\Z$. In the space $L^2(-1/2,1/2)$ the operator $A$ corresponds to the operator of multiplication with $f_0(\xi) := e^{-\delta\xi^2}$ and $\frake_i$ corresponds to $e_i = e^{2\pi{\rm i}i\cdot}$. For a Borel set $\Delta\subset [e^{-\delta/4},1]$ we set $\Delta' := (-\frac 1\delta\log\Delta)^{1/2}$ and $\Delta'' := (-\Delta')\cup\Delta'$. Then $E(\Delta)$ corresponds to the operator of multiplication with $\chi_{\Delta''}$. Hence,
$$
\mu_{\frake_i}(\Delta) = \<\chi_{\Delta''}e_i,e_i\> = \int_{\Delta''}|e_i|^2\,d\xi = \Leb_1(\Delta'') = 2\Leb_1(\Delta').
$$
This implies that $\mu_{\frake_i}$ is absolutely continuous with respect to $\Leb_1$. However, since the square root is not Lipschitz continuous on $[0,1/4]$, the measure $\mu_{\frake_i}$ is {\it not} Lipschitz continuous with respect to $\Leb_1$ on any set $(1-\veps,1)$ so that Proposition \ref{p:sa_suff} is not applicable. In fact, the next proposition shows that $(A^k\frake_i)_{k\in\N}$ is not a Bessel sequence.

\begin{prop}
For each $i\in\Z$ the sequence $(A^k\frake_i)_{k\in\N}$ is not a Bessel sequence.
\end{prop}
\begin{proof}
For $\veps\in (0,1-e^{-\delta/4}]$ (i.e., $1-\veps\in [e^{-\delta/4},1)$) let $\Delta_\veps := (1-\veps,1)$. Then
$$
\Delta_\veps' = \left(-\frac 1 \delta\log\Delta_\veps\right)^{1/2} = \left[0,-\frac 1 \delta\log(1-\veps)\right]^{1/2} = \left[0,\frac 1 {\sqrt\delta}\left(\log\frac 1 {1-\veps}\right)^{1/2}\right].
$$
Thus,
$$
\mu_{\frake_i}(\Delta_\veps) = 2\Leb_1(\Delta_\veps') = \frac 2 {\sqrt\delta}\left(\log\frac 1 {1-\veps}\right)^{1/2}.
$$
Suppose that $(A^k\frake_i)_{k\in\N}$ is a Bessel sequence. Then, by Theorem \ref{t:sa_charac}, there exists some $C > 0$ such that $\mu_{\frake_i}(\Delta_\veps)\le C\veps$ for each $\veps\in (0,1)$. Hence, there exists $C > 0$ such that $\log((1-\veps)^{-1})\le C\veps^2$ for each $\veps\in (0,1-e^{-\delta/4}]$. But then $C\ge\veps^{-2}\log((1-\veps)^{-1})\to\infty$ as $\veps\to 0$. A contradiction.
\end{proof}

We conclude that, using a sampling scheme $\calS$ as described above, the only chance for a function $f$ to be recovered in a stable way is when for each $i\in\Z$ the set $\{k\in\N : (k\delta,mi)\in\calS\}$ is finite.

\section{Auxiliary Statements}\label{s:aux}
In the sequel we collect some lemmas from measure theory and operator theory which we make use of in the preceding sections.

\subsection{Measure Theory}\label{ss:aux_meas}
Throughout this subsection, let $(\Omega,\Sigma)$ be a measurable space. We recall the definition of Lipschitz continuity of set functions (see Definition \ref{d:lipschitz}).

\begin{lem}\label{l:liphoelder}
Let $\mu$ and $\nu$ be finite positive measures on $(\Omega,\Sigma)$. Then the following statements are equivalent.
\begin{enumerate}
\item[{\rm (i)}]  $\nu$ is Lipschitz continuous with respect to $\mu$ with constant $C$.
\item[{\rm (ii)}] $\nu\ll\mu$ such that $f = d\nu/d\mu\in L^\infty(\mu)$ with $\|f\|_{L^\infty(\mu)}\le C$.
\end{enumerate}
\end{lem}
\begin{proof}
If $\nu$ is Lipschitz continuous with respect to $\mu$ with constant $C$ then $\nu\ll\mu$ is obvious. Let $f = d\nu/d\mu$ and put $\Delta_n := \{\omega : f(\omega)\ge C+1/n\}$. Then
$$
C\mu(\Delta_n)\,\ge\,\nu(\Delta_n) = \int_{\Delta_n}f\,d\mu\,\ge\,\left(C + \frac 1 n\right)\mu(\Delta_n),
$$
that is, $\mu(\Delta_n) = 0$. Hence, $f\le C$ $\mu$-a.e. which shows $\|f\|_{L^\infty(\mu)}\le C$. The implication (ii)$\Sra$(i) is obvious.
\end{proof}

Clearly, if $\nu,\mu_1,\mu_2$ are measures and $\nu$ is $\sqrt{\mu_1\mu_2}$-Lipschitz continuous, then $\nu$ is absolutely continuous with respect to both measures $\mu_1$ and $\mu_2$. The following lemma shows in particular that the corresponding density functions lie in $L^2$.

\begin{lem}\label{l:lipschitz}
Let $\nu,\mu_1,\mu_2$ be finite positive measures on $(\Omega,\Sigma)$. Then the following statements are equivalent.
\begin{enumerate}
\item[{\rm (i)}]   $\nu$ is $\sqrt{\mu_1\mu_2}$-Lipschitz continuous with constant $C$.
\item[{\rm (ii)}]  $\nu\ll\mu_1$, $d\nu/d\mu_1\in L^2(\mu_1)$, and for any $\Delta\in\Sigma$ we have
$$
\left\|(d\nu/d\mu_1)|\Delta\right\|_{L^2(\mu_1)}^2\le C^2\mu_2(\Delta).
$$
\item[{\rm (iii)}] $\nu\ll\mu_2$, $d\nu/d\mu_2\in L^2(\mu_2)$, and for any $\Delta\in\Sigma$ we have
$$
\left\|(d\nu/d\mu_2)|\Delta\right\|_{L^2(\mu_2)}^2\le C^2\mu_1(\Delta).
$$
\end{enumerate}
\end{lem}
\begin{proof}
(i)$\Sra$(ii). It is clear that $\nu\ll\mu_1$. Since the measures are finite, we have that $f_1 :=d\nu/d\mu_1\in L^1(\mu_1)$. Let $\Delta\in\Sigma$ be arbitrary. For a simple function $g = \sum_{k=1}^n\alpha_k\chi_{\Delta_k}$, $0\le g\le f_1$, with mutually disjoint $\Delta_k\in\Sigma$, $\Delta_k\subset\Delta$, we have $gf_1\in L^1(\mu_1)$ since $g$ is bounded. Moreover,
\begin{align*}
\int gf_1\,d\mu_1
&= \int g\,d\nu = \sum_{k=1}^n\alpha_k\nu(\Delta_k)\,\le\,C\sum_{k=1}^n\alpha_k\sqrt{\mu_1(\Delta_k)\mu_2(\Delta_k)}\\
&\le\,C\left(\sum_{k=1}^n\alpha_k^2\,\mu_1(\Delta_k)\right)^{1/2}\left(\sum_{k=1}^n\mu_2(\Delta_k)\right)^{1/2}\\
&\le C\left(\int g^2\,d\mu_1\right)^{1/2}\sqrt{\mu_2(\Delta)}\\
&\le C\left(\int gf_1\,d\mu_1\right)^{1/2}\sqrt{\mu_2(\Delta)}.
\end{align*}
That is,
$$
\int gf_1\,d\mu_1\,\le\,C^2\mu_2(\Delta).
$$
Let $(g_n)$ be a sequence of non-negative simple functions on $\Delta$ converging pointwise to $f_1|\Delta$ from below. Then from the above inequality and the monotone convergence theorem we obtain $f_1\in L^2(\mu_1)$ and $\|f_1|\Delta\|_{L^2(\mu_1)}^2\le C^2\mu_2(\Delta)$.

(ii)$\Sra$(i). As above, put $f_1 :=d\nu/d\mu_1$. For any $\Delta\in\Sigma$ we have
$$
\nu(\Delta) = \int_\Delta f_1\,d\mu_1\,\le\,\left(\int_\Delta f_1^2\,d\mu_1\right)^{1/2}\sqrt{\mu_1(\Delta)}\,\le\,C\sqrt{\mu_1(\Delta)\mu_2(\Delta)}.
$$
This is (i). The equivalence (i)$\Slra$(iii) can be proved similarly.
\end{proof}

The {\it Poisson kernel}\label{page:poisson} $P : (\C\setminus\T)\times\T\to\C$ is defined by $P(w,z) := \Re\frac{z+w}{z-w} = \frac{1-|w|^2}{|z-w|^2}$. We have $P(\ol w^{-1},z) = -P(w,z)$. If $\mu$ is a complex Borel measure on $\T$, one defines the {\it Poisson integral} of $\mu$ by
$$
\calP[\mu](w) := \int_\T P(w,z)\,d\mu(z),\qquad w\in\C\setminus\T.
$$
We note that $\calP[\arc](w) = 1$ for each $w\in\C\setminus\T$. The following lemma contains the so-called Stieltjes' inversion formula (for the circle), see, e.g., \cite[Satz 1.1.6]{wo}.

\begin{lem}\label{l:stieltjes}
Let $\mu$ be a complex measure on $\T$ and $\Delta$ an open arc in $\T$ with endpoints $\alpha,\beta\in\T$. Then
$$
\lim_{r\upto 1}\,\int_{\Delta}\calP[\mu](rz)\,d|z| = \mu(\Delta) + \frac 1 2\mu(\{\alpha,\beta\}).
$$
\end{lem}

\subsection{Operator Theory}\label{ss:aux_op}
Let $A$ be a bounded normal operator and $E$ its spectral measure. Recall that for $x\in\calH$ we defined the space
$$
\calH_x := \ol\linspan\left\{E(\Delta)x : \Delta\in\frakB\right\}.
$$
The following two lemmas should essentially be well known. However, since we couldn't find any of them in the literature, we state and prove them here.

\begin{lem}\label{l:set_equal}
Let $A$ be a bounded normal operator in $\calH$ and $x\in\calH$. Then we have
$$
\ol\linspan\left\{A^kx,\,(A^*)^kx : k\in\N\right\} = \calH_x.
$$
\end{lem}
\begin{proof}
Let $k\in\N$ and let $(f_n)$ be a sequence of simple functions converging uniformly to $f(z) = z^k$ on $\sigma(A)$. Then
$$
A^kx = \int \la^k\,dE_\la x = \lim_{n\to\infty}\int f_n\,dEx\,\in\,\calH_x.
$$
Similarly, one shows that $(A^*)^kx\in\calH_x$. For the converse inclusion, we first note that by the Stone-Weierstra\ss\ Theorem, the set $\{z\mapsto p(z) + q(\ol z) : p,q\text{ polynomials}\}$ is dense in $C(\sigma(A))$. Therefore,
$$
\ol\linspan\left\{A^kx,\,(A^*)^kx : k\in\N\right\} = \ol\linspan\{f(A)x : f\in C(\sigma(A))\}.
$$
Let $\Delta\subset\C$ be a closed rectangle and choose a uniformly bounded sequence $(f_n)\subset C(\sigma(A))$ such that $f_n\upto\chi_\Delta$ pointwise. Then it is well known that $f_n(A)x\to E(\Delta)x$ as $n\to\infty$. Thus, $E(\Delta)x\in\ol\linspan\{A^kx,\,(A^*)^kx : k\in\N\}$. Now, it is easy to see that the system
$$
\calD := \left\{\Delta\in\frakB : E(\Delta)x\in\ol\linspan\{A^kx,\,(A^*)^kx : k\in\N\}\right\}
$$
is a Dynkin system. And since it contains the $\pi$-system of the closed rectangles, it coincides with $\frakB$.
\end{proof}

Recall that for $x,y\in\calH$ we defined the measures $\mu_{yx}$ and $\mu_x$ by $\mu_{yx}(\Delta) = \<E(\Delta)y,x\>$ and $\mu_x(\Delta) = \<E(\Delta)x,x\>$, $\Delta\in\frakB$.

\begin{lem}\label{l:mu_yx}
Let $A$ be a bounded normal operator in $\calH$ and $x,y\in\calH$. Then the following statements hold.
\begin{enumerate}
\item[{\rm (i)}]  We have
\begin{equation}\label{e:mu_yx}
|\mu_{yx}|(\Delta)\,\le\,\sqrt{\mu_y(\Delta)\mu_x(\Delta)},\qquad\Delta\in\frakB,
\end{equation}
and $f_y := d\mu_{yx}/d\mu_x\in L^2(\mu_x)$ satisfies
$$
\int_\Delta f_y\,dEx = P_xE(\Delta)y,\qquad\Delta\in\frakB,
$$
where $P_x$ denotes the orthogonal projection onto $\calH_x$.
\item[{\rm (ii)}] The operator $L_x : \calH_x\to L^2(\mu_x)$, defined by
$$
L_xy := f_y,\qquad y\in\calH_x,
$$
is an isometric isomorphism between $\calH_x$ and $L^2(\mu_x)$.
\end{enumerate}
\end{lem}
\begin{proof}
(i). Let $\Delta\in\frakB$. Then (``m.d.'' standing for ``mutually disjoint'')
$$
|\mu_{yx}|(\Delta) = \sup\left\{\sum_{i=1}^n\left|\left\<E(\Delta_i)y,E(\Delta_i)x\right\>\right| : \Delta_i\in\frakB(\Delta)\text{ m.d.}\right\}.
$$
Applying the Cauchy-Schwarz inequality consecutively on both scalar products $\aproduct$ and $\aproduct_{\C^n}$, we obtain
\begin{align*}
|\mu_{yx}|(\Delta)
&\le \sup\left\{\left(\sum_{i=1}^n\mu_y(\Delta_i)\right)\left(\sum_{i=1}^n\mu_x(\Delta_i)\right) : \Delta_i\in\frakB(\Delta)\text{ m.d.}\right\}^{1/2}\\
&= \sqrt{\mu_y(\Delta)\mu_x(\Delta)},
\end{align*}
which is \eqref{e:mu_yx}. This and Lemma \ref{l:lipschitz} imply $f_y = d\mu_{yx}/d\mu_x\in L^2(\mu_x)$. We have
$$
\<E(\Delta)y,x\> = \mu_{yx}(\Delta) = \int_\Delta f_y\,d\mu_x = \<(\chi_\Delta f_y)(A)x,x\>,\qquad\Delta\in\frakB.
$$
From this it is easily seen that also
$$
\<E(\Delta)y,z\> = \<(\chi_\Delta f_y)(A)x,z\>,\qquad\Delta\in\frakB,
$$
for all $z\in\calH_x$. This proves
$$
P_xE(\Delta)y = (\chi_\Delta f_y)(A)x = \int_\Delta f_y\,dEx.
$$
(ii). The linearity of the operator $L_x$ is inherited from the linearity of $y\mapsto\mu_{yx}$. If $y\in\calH_x$, we observe that
$$
\|L_x y\|^2 = \int |f_y|^2\,d\mu_x = \left\|\int f_y\,dEx\right\|^2 = \|P_xy\|^2 = \|y\|^2.
$$
Hence, $L_x$ is isometric. In order to see that $L_x$ is surjective, let $f\in L^2(\mu_x)$ such that $\<f,f_y\> = 0$ for all $y\in\calH_x$. Then $x\in\dom(f(A))$ and $\<f(A)x,y\> = \int f\,d\mu_{xy} = \int f\ol{f_y}\,d\mu_x = \<f,f_y\> = 0$ for all $y\in\calH_x$. Thus, $f(A)x = 0$. 
This implies
$$
\|f\|_{L^2(\mu_x)}^2 = \int |f|^2\,d\mu_x = \left\|\int f\,dEx\right\|^2 = \|f(A)x\|^2 = 0,
$$
and thus $f = 0$.
\end{proof}

\
\\
{\bf Acknowledgement.} The author gratefully acknowledges support from MinCyT Argentina under grant PICT-2014-1480. Moreover, he would like to thank A. Aldroubi, C. Cabrelli, U. Molter, and V. Paternostro for fruitful discussions.

\end{document}